\documentclass[11pt, regno]{amsart}
\usepackage[cp1251]{inputenc}
\usepackage[russian, english]{babel} 

\usepackage{amsmath}
\usepackage{amssymb}
\usepackage{amsthm}
\usepackage{amscd}
\usepackage{amsbib}

\usepackage[matrix,arrow,curve]{xy}

\begin{document}
\frenchspacing
\righthyphenmin=2

\newtheorem{lm}{Lemma}
\newtheorem*{lm*}{Lemma}
\newtheorem{prop}{Proposition}
\newtheorem{tm}{Theorem}
\newtheorem*{tm*}{Theorem}
\newtheorem{cl}{Corollary}
\newtheorem*{cl*}{Corollary}
\theoremstyle{definition}
\newtheorem{df}{Definition}
\newtheorem{rem}{Remark}

\newcommand{\Hom}{\mathop{\mathrm{Hom}}\nolimits}
\newcommand{\Br}{\mathop{\mathrm{Br}}\nolimits}
\newcommand{\Ch}{\mathop{\mathrm{Ch}}\nolimits}
\newcommand{\Gm}{\mathop{\mathbb{G}_m}\nolimits}
\newcommand{\Ker}{\mathop{\mathrm{Ker}}\nolimits}
\newcommand{\Coker}{\mathop{\mathrm{Coker}}\nolimits}
\newcommand{\Vect}{\mathop{\mathrm{Vect}}\nolimits}
\newcommand{\Mor}{\mathop{\mathrm{Mor}}\nolimits}
\newcommand{\Rep}{\mathop{\mathrm{Rep}}\nolimits}
\newcommand{\GL}{\mathop{\mathrm{GL}}\nolimits}
\newcommand{\SL}{\mathop{\mathrm{SL}}\nolimits}
\newcommand{\Sp}{\mathop{\mathrm{Sp}}\nolimits}
\newcommand{\End}{\mathop{\mathrm{End}}\nolimits}
\newcommand{\Tor}{\mathop{\mathrm{Tor}}\nolimits}
\newcommand{\E}{\mathop{\mathrm{E}}\nolimits}
\newcommand{\F}{\mathop{\mathrm{F}}\nolimits}
\newcommand{\Hh}{\mathop{\mathrm{H}}\nolimits}
\newcommand{\K}{\mathop{\mathrm{K}}\nolimits}
\newcommand{\Z}{{\mathbb Z}}
\newcommand{\Q}{{\mathbb Q}}
\newcommand{\m}{{\mathfrak m}}
\newcommand{\ag}{{\mathfrak a}}
\newcommand{\Gal}{\mathop{\mathrm{Gal}}\nolimits}
\newcommand{\chr}{\mathop{\mathrm{char}}\nolimits}
\newcommand{\di}{\mathop{\mathrm{dim}}\nolimits}

\author{Maria Yakerson}

\address{University of Duisburg-Essen\\
Faculty of Mathematics\\
Thea-Leymann-Str. 9 45127\\
Essen\\
Germany}

\email{mura.yakerson@gmail.com}

\dedicatory{To Sasha Ivanov}

\title[K-theory of $\SL_{2n} / \Sp_{2n}$, $\E_6 / \F_4$ and their twisted forms]{Algebraic K-theory of varieties $\SL_{2n} / \Sp_{2n}$, $\E_6 / \F_4$ and their twisted forms}



\begin{abstract}
Let $\SL_{2n}$, $\Sp_{2n}$, $\E_6 = G^{sc}(\E_6)$, $\F_4 = G(\F_4)$ be simply connected split algebraic groups over an arbitrary field $F$. Algebraic \linebreak K-theory of affine homogeneous varieties $\SL_{2n} / \Sp_{2n}$ and $\E_6 / \F_4$ is computed. Moreover, explicit elements that generate $K_*(\SL_{2n} / \Sp_{2n})$ and $K_*(\E_6/\F_4)$ as $K_*(F)$-algebras are provided. For some twisted forms of these varieties K-theory is also computed.
\end{abstract}

\thanks{This work has been supported by the grant Sonderforschungsbereich Transregio 45}

\maketitle

\section*{Introduction}

Algebraic K-theory is already known for some classes of algebraic varieties. At first it was computed for Severi-Brauer varieties by D.\,Quillen \cite{quil} and for smooth projective quadrics by R.\,Swan \cite{swan}. Then M.\,Levine \cite{lev} computed the K-theory of split semisimple simply connected algebraic groups. I.\,Panin \cite{pan} generalized this computation for all semisimple simply connected algebraic  groups and computed the K-theory of flag varieties (see~\cite{panin}). Later A.\,Ananyevskiy \cite{an} computed the K-theory of homogeneous varieties $G/H$, where $H \subset G$ are connected reductive algebraic groups of the same rank. In all these cases K-theory turned out to be isomorphic to a sum of K-theories of some central semisimple algebras.

We provide a computation of the K-theory for affine homogeneous varieties $\SL_{2n} / \Sp_{2n}$ and $\E_6 / \F_4$. The computation is based on using the Merkurjev spectral sequence for the equivariant K-theory (see~\cite{mer}). The key point which allows us to accomplish the computation is the following fact: for the chosen varieties $G/H$ there is an epimorphism $i^* \colon R(G) \to R(H)$ on the rings of representations, and its kernel is generated by explicit elements. Here can be seen a big difference with the case of $G/H$ where $G$ and $H$ have the same rank. In that case A.\,Ananyevskiy has shown~\cite[Theorem 2]{an} that $R(H)$ is a free $R(G)$-module.

The following theorem is proved.

\begin{tm*}
\label{maintm}
There are isomorphisms of graded $K_*(F)$-modules:
$$
K_*(\SL_{2n}/\Sp_{2n}) \simeq K_*(F) \otimes \Lambda (\Z^{n-1});
$$
$$
K_*(\E_6/\F_4) \simeq K_*(F) \otimes \Lambda (\Z^2),
$$
where $\Lambda(\Z^m)$ is an exterior algebra considered with the natural grading. 
\end{tm*}

Moreover, we provide the elements $t_1, \dots, t_{n-1} \in K_1(\SL_{2n}/\Sp_{2n})$ and $s_1, s_2 \in K_1(\E_6/\F_4)$ that are multiplicative generators of $K_*(F)$-algebras $K_*(\SL_{2n}/\Sp_{2n})$ and $K_*(\E_6/\F_4)$ respectively. These elements are constructed similarly to those in topological 
K-theory of these varieties (see~\cite{mi}). The proof is based on M.\,Levine's computation \cite[Theorem 2.1]{lev} of multiplicative generators for $K_*(\SL_{2n})$ and $K_*(\E_6)$ as algebras over $K_*(F)$.

Explicitly constructed isomorphisms in the split case allow to compute K-theory of some twisted forms of these varieties using Panin's splitting principle \cite{pan}.

\begin{tm*}
\label{twisttm}
Assume $\chr(F) \ne 2$. Let $\gamma \colon \Gal(F^{sep} / F) \to (\Sp_{2n} / \mu_2)(F^{sep})$ be a 1-cocycle, $A = \End(V)$ where $V$ is a $2n$-dimensional vector space over $F$, and $\tau$ the standard symplectic involution on $A$. Denote $B_i$ the central simple algebra $A_{\gamma}$ for $i$ odd, and $F$ for $i$ even $(1 \leqslant i \leqslant n-1)$. Denote $B_I = B_{i_1} \otimes \dots \otimes B_{i_q}$ for every $I = \{i_1 < \dots < i_q \} \subseteq \{1, \dots, n-1 \}$.
Then the following graded $K_*(F)$-modules are isomorphic:
$$
K_*(\SL_{1, A_{\gamma}} / \Sp(A_{\gamma}, \tau_{\gamma}) ) \,\, \simeq \,\,   \bigoplus_{I \subseteq \{1, \dots, n-1\}} \! \! \! \! \! K_{*-|I|}(B_I).
$$
Let $\delta \colon \Gal(F^{sep} / F) \to \F_4(F^{sep})$ be a 1-cocycle. Then the following graded $K_*(F)$-modules are isomorphic:
$$
K_*((\E_6 / \F_4)_{\delta}) \,\, \simeq \,\, \bigoplus_{I \subseteq \{1, 2\}} \! \! \! K_{*-|I|}(F).
$$
\end{tm*}

In section 1 we construct multiplicative generators of K-theory and introduce some notation. In section 2 we study the Merkurjev spectral sequence which is used in section 3 to compute the K-theories of the varieties in question as graded modules over the K-theory of a base field. In section 4 we compute the multiplicative generators of the K-theory and state the answer in the split case. In section 5 we state the problem for twisted forms of the varieties. Then in section 6 we describe how to twist the multiplicative generators with a cocycle. Finally, in section 7 we show how Panin's splitting principle helps to reduce the problem to the split case, which is already solved.

{\it Acknowledgement}. The author wishes to express her sincere gratitude to A.\,Ananyevskiy and I.\,Panin for stating the problem and for numerous discussions. The author is also grateful to V.\,Sosnilo, S.\,Sinchuk and 
A.\,Lavrenov for helpful comments on the subject of this paper.

\section{Construction of the generators for $K_1(G/H)$}
\label{construct}

\subsection{Representation rings of $\SL_{2n}$ and $\Sp_{2n}$, $\E_6$ and $\F_4$}
\label{fund_rep}

\begin{df}
Let $G$ be an algebraic group over a field $F$. The {\it representation ring} $R(G)$ of a group $G$ is the Grothendieck group of the category $\Rep_F(G)$ with multiplication induced by tensor product of representations.
\end{df}

Suppose we have a subgroup $\, i \colon H \hookrightarrow G$. Then restriction of representations induces the homomorphism $i^* \colon R(G) \to R(H)$.

\subsubsection*{$\SL_{2n}$ and $\Sp_{2n}$}
\label{rep1}

Denote the vector representation by $V$. Then for representation rings of the groups $\SL_{2n}$ and $\Sp_{2n}$ we have:
$$
R(\SL_{2n}) = \Z[V, \Lambda^2 V, \dots, \Lambda^{2n-1} V], \quad
R(\Sp_{2n}) = \Z[V, \Lambda^2 V, \dots, \Lambda^n V]. 
$$
The representations $\Lambda^k V$ and $\Lambda^{2n-k} V$ become isomorphic after restriction to $\Sp_{2n}$ for every $k = 1, \dots, n-1$. 
The homomorphism $i^* \colon R(\SL_{2n}) \to R(\Sp_{2n})$ is surjective.
The ideal $\Ker \, i^*$ is generated by elements ${\Lambda^{k} V - \Lambda^{2n-k} V}$, where $ k = 1, \dots, n-1.$

\subsubsection*{$\E_6$ and $\F_4$}
\label{rep2}

Let $\rho$ and $\rho^{\vee}$ be the 27-dimensional fundamental representations of $\E_6$, and let $\rho'$ be the 26-dimensional fundamental representation of $\F_4$. Then for representation rings of the groups $\E_6$ and $\F_4$ we have:
$$R(\E_6) = \Z[\rho, \rho^{\vee}, \Lambda^2 \rho, \Lambda^2 \rho^{\vee}, \Lambda^3 \rho, \mathrm{Ad}_{\E_6}] , \quad
R(\F_4) = \Z[\rho', \Lambda^2 \rho', \Lambda^3 \rho', \mathrm{Ad}_{\F_4}],$$
and $\Lambda^3 \rho \simeq \Lambda^3 \rho^{\vee}$.
The representations $\rho$ and $\rho^{\vee}$ become isomorphic after restriction to $F_4$. It is known  that $i^*(\rho) = i^*(\rho^{\vee}) = \rho' + 1 \,$; $\, i^*(\mathrm{Ad}_{\E_6}) = \rho' + \mathrm{Ad}_{\F_4}$ \cite[p.~298]{kay}. Hence the homomorphism $i^*$ is surjective.
The ideal $\Ker \, i^*$ is generated by elements $\rho - \rho^{\vee}$ and $\Lambda^2 \rho - \Lambda^2 \rho^{\vee}$.

\subsection{Construction}
\label{construction}

Suppose we have an affine homogeneous variety $G/H$. Assume there are two nonisomorphic representations of the group $G$ that are isomorphic when restricted to the subgroup $H$. 
In other words, there are homomorphisms $\phi, \psi \colon G \to \GL(F^k) $ and a matrix $\alpha \in \GL(F^k)$ such that
$
\phi(h) = \alpha^{-1} \psi(h) \alpha \, \, \, \, \, \forall h \in H.
$

Using these data we construct a well-defined map $\chi$ from $G/H$ to $\GL(F^k)$: 
$[g] \mapsto \phi(g) \alpha^{-1} \psi(g)^{-1} \alpha$. 
We identify $\Mor_{F}(G/H, \GL_k)$ with $\GL_k(F[G/H])$ and consider the composition: 
$$ 
\xymatrix{
 \GL_k(F[G/H]) \ar@{->>}[r] \ar@{^{(}->}[r] & \GL(F[G/H]) \ar@{->>}[r] & K_1(F[G/H]) \ar@{=}[r] & K_1(G/H). \\
}
$$
This way, the map $\chi$ gives us an element in $K_1(G/H)$. It is denoted by $\beta(\phi - \psi)$ and defined by the following formula:
$$
\beta(\phi - \psi) = \bigl[\,[g] \mapsto \phi(g) \alpha^{-1} \psi(g)^{-1} \alpha \bigr] \in  K_1(G/H); \quad  [g] \in G/H.
$$

\subsection{Application}
\label{elements}
Now we will provide some elements of $K_1(\SL_{2n}/\Sp_{2n})$ and $K_1(\E_6/\F_4)$, and later we will show that they are multiplicative ge\-nerators of $K_*(F)$-algebras $K_*(\SL_{2n}/\Sp_{2n})$ and $K_*(\E_6/\F_4)$. These varieties are affine as quotients of reductive groups by reductive subgroups (see~\cite{rich}), so we can apply here the construction described in \ref{construction}.

For the group $\SL_{2n}$ consider the vector representation $V$ and its exterior powers $\Lambda^{k} V$. 
For every $1 \leqslant k \leqslant n-1$ the representations $\Lambda^{k} V$ and $\Lambda^{2n-k} V$ are isomorphic when restricted to $\Sp_{2n}$ (see~\ref{rep1}). The corresponding elements of $K_1(\SL_{2n}/\Sp_{2n})$ are defined as follows: 
$$
t_k = \beta(\Lambda^{k} V - \Lambda^{2n-k} V), \, \, 1 \leqslant k \leqslant n-1.
$$

For the group $\E_6$ consider the fundamental representations $\rho$ and $\rho^{\vee}$, which are isomorphic when restricted to $\F_4$ (see~\ref{rep2}). Here are the desired elements of $K_1(\E_6/\F_4)$: 
$$
s_1 = \beta (\rho - \rho^{\vee}); \, \, \, s_2 = \beta (\Lambda^2 \rho - \Lambda^2 \rho^{\vee}).
$$

\subsection{Notation} 
\label{defs}
Here we introduce some notation that will be used later.
\begin{itemize}
\item $G/H$ (or $X$) --- both varieties $\SL_{2n} / \Sp_{2n}$ and $\E_6 / \F_4$;
\item $\rho_1, \dots, \rho_l$ --- fundamental representations of the group $G$;
\item $\{(\rho_{i_1}, \rho_{i_2})\}_{i=1}^m$ --- pairs of fundamental representations of $G$, that are isomorphic when restricted to $H$ ($m=n-1$ or $m=2$);
\item $\widehat{\rho_i} = \rho_{i_1} - \rho_{i_2}$ --- elements of $R(G)$ that generate $\Ker\,i^*$ (see~\ref{fund_rep}). 
\end{itemize}

\section{Merkurjev spectral sequence}
\label{spect}
The Merkurjev spectral sequence allows to express the K-theory of a variety $X$ equipped with an action of an algebraic group $G$ in terms of the $G$-equivariant K-theory of $X$ (see~\cite{mer},~\cite{merk}).

\begin{df}
Let $X$ be a variety equipped with an action of an algebraic group  $G$. The {\it $G$-equivariant K-theory} of $X$ is the K-theory of the category of $G$-equivariant vector bundles over $X$. It is denoted by $K_*(G; X)$.
\end{df}

For computing $K_*(G/H)$ as a $K_*(F)$-module we will need the following theorem of A.\,Merkurjev \cite[Theorem 4.3]{mer}.

\begin{tm*}[Merkurjev]
Let $G$ be a split reductive group such that $\pi_1(G)$ is torsion-free, and let $X$ be a $G$-scheme. Then there is a spectral sequence:
$$
E_{p, q}^2 = \Tor_p^{R(G)}(\Z, K_q(G; X)) \Longrightarrow K_{p+q}(X).
$$
\end{tm*}

Since both groups $G = \SL_{2n}$ and $G = \E_6$ are simple and simply connected, their fundamental groups are trivial \cite[Cor.~1.3]{mer}. 
Applying this theorem to the case of the variety $G/H$, on which the group $G$ acts by left translation, we get the following spectral sequence:
$$
E_{p, q}^2 = \Tor_p^{R(G)}(\Z, K_q(G; G/H)) \Longrightarrow K_{p+q}(G/H).
$$
Let us compute the terms of its second page.
\subsection{Computation of $E_{p, q}^2$}
\label{second_term}

\begin{lm}
\label{equiv}
$K_i(G; G/H) \simeq R(H) \otimes K_i(F)$ as $R(H)$-modules.
\end{lm}
\begin{proof}
This statement is proved in~\cite[Lemma 9]{an}. The proof is based on the fact that the categories $\Vect^G(G/H)$ and $\Rep(H)$ are equivalent \cite[Example~2]{merk}.
\end{proof}

Therefore we need to compute the second page of the following spectral sequence:
 $$E_{p, q}^2 = \Tor_p^{R(G)}(\Z, R(H) \otimes K_q(F)) \Longrightarrow K_{p+q}(G/H).$$
At first, we will treat the case $q=0$.
 
\subsubsection*{Computation of $\Tor_p^{R(G)}(\Z, R(H))$}

First let us notice that $\Z$ is considered as an $R(G)$-module by means of the dimension homomorphism $R(G) \to \Z$, and $R(H)$ becomes an $R(G)$-module by means of the homomorphism $i^* \colon R(G) \to R(H)$. Recall that for both considered varieties $G/H$ the homomorphism $i^*$ is surjective (see~\ref{fund_rep}).

We can see that the sequence $(\widehat{\rho_1}, \dots, \widehat{\rho_m})$ is regular in $R(G)$. Hence we can write the corresponding Koszul resolution $\K_{\bullet} \to R(H)$, consisting of free $R(G)$-modules:
$$
\xymatrix{
\Lambda^m(R(G)^m) \ar@{^{(}->}[r]^{\qquad d_m} & \dots \ar[r] & \Lambda^2(R(G)^m) \ar[r]^{~d_2} & R(G)^m \ar[r]^{~d_1} & 
R(G) \ar@{->>}[r]^{~i^*} & R(H) \\
}
$$
Let $e_i$ generate $R(G)^m$ as a free $R(G)$-module ($i=1 \dots m$), then the differentials are defined the following way:
 $d_1 \colon e_i \mapsto \widehat{\rho_i}$; 
$ \quad d_2 \colon e_i \wedge e_j \mapsto \linebreak \widehat{\rho_i} \cdot e_j - \widehat{\rho_j} \cdot e_i; \,\, $ etc.

Consider the isomorphism $ R(H) \otimes_{R(G)} \Z \xrightarrow{\sim} \Z$:
 $ \quad \rho \otimes n \mapsto \di(\rho) \cdot n$. Let us multiply the resolution $\K_{\bullet}$ termwise by $\Z$ and apply this isomorphism:
$$
 - \otimes_{R(G)} \Z \colon \qquad \Lambda^m(\Z^m) \to \dots  \to \Lambda^2(\Z^m) \to \Z^m \to \Z \to \Z \to 0
$$
All the differentials will become zero because $\di (\rho_{i_1}) = \di (\rho_{i_2})$, and so $\di(\widehat{\rho_i})=0$ for every $i$.

As a result, we get the formula:
\begin{equation}
\label{comp_tor}
\Tor_p^{R(G)}(\Z, R(H)) =\Hh_p(\K_{\bullet} \otimes_{R(G)} \Z) = \Lambda^p(\Z^m).
\end{equation}

\subsubsection*{Final presentation of $E_{p, q}^2$}
   
To finish the computation of $E_{p, q}^2$ we will need the following lemma.

\begin{lm}
\label{tor3}
$\Tor_p^{R(G)}(\Z, R(H) \otimes K_i(F)) = \Tor_p^{R(G)}(\Z, R(H)) \otimes K_i(F)$ for every $i \geqslant 0$.
\end{lm}
\begin{proof}
Because of associativity of tensor product we have:
$$(\Z \otimes_{R(G)} R(H)) \otimes K_i(F) = \Z \otimes_{R(G)} (R(H) \otimes K_i(F)).$$

This implies the existence of two spectral sequences that converge to triple $\Tor$:
$$
 \widetilde E_2^{p, q} = \Tor_p^{\Z}(\Tor_q^{R(G)}(\Z, R(H)), K_i(F)) \Longrightarrow \Tor_{p+q}(\Z, R(H), K_i(F)) ,
$$
$$
\widehat E_2^{p, q} = \Tor_p^{R(G)}(\Z, \Tor_q^{\Z}(R(H), K_i(F))) \Longrightarrow \Tor_{p+q}(\Z, R(H), K_i(F)) .
$$

Observe that:
 $\widetilde E_2^{p, q} = 0$ for $p \not = 0$ because $\Tor_q^{R(G)}(\Z, R(H))$ is a free $\Z$-module
(see~(\ref{comp_tor})); 
 $\widehat E_2^{p, q} = 0$ for $q \not = 0$ because $R(H)$ is a free $\Z$-module.
Therefore both spectral sequences degenerate on the second page and $\widetilde E_2^{0, p} = \widehat E_2^{p, 0}$, which is indeed the statement of the lemma.

\end{proof}

Lemma \ref{equiv}, formula (\ref{comp_tor}) and Lemma \ref{tor3} imply that the Merkurjev spectral sequence for the varieties $G/H = \SL_{2n} / \Sp_{2n}$ and $G/H = \E_6 / \F_4$ looks this way:
$$
E_{p, q}^2 = \Lambda^p(\Z^m) \otimes K_q(F) \Longrightarrow K_{p+q}(G/H),
$$
where $m = rk (G) - rk (H)$.
The spectral sequence is first-quadrant, its differential $d_{p,q}^2$ acts from $E_{p,q}^2$ to $E_{p-2,q+1}^2$.

\subsection{Degeneration of $E_{p,q}^*$}
The Merkurjev spectral sequence is a special case of the Levine spectral sequence \cite[3.1]{mer}. There is a multiplicative structure on the zero row of the second page of the Levine spectral sequence which is denoted by $\smile_2$ \cite[Section~1]{lev}. To check that the spectral sequence
$E_{p,q}^*$ degenerates we will need the following technical lemma.

\begin{lm}
\label{multi}
The multiplicative structure $\smile_2$ on $\bigoplus_p E_{p, 0}^2$ coincides with the natural product on $\bigoplus_p \Lambda^p(\Z^m)$.
\end{lm}
\begin{proof}
The following statement is true \cite[Example 1.1]{lev}: let $R$ be a local ring, $\m$ its maximal ideal, $x_1, \dots, x_n$ a regular sequence in $\m$ and $B$ an ideal in $R$. Then the multiplicative structure $\smile_2$ on $\oplus_p \Tor_p^R(R/(x_1, \dots, x_n), R/B)$ coincides with the natural product.

Under the conditions of the lemma we need to show that the two products on $\oplus_p \Lambda^p(\Z^m) = \oplus_p \Tor_p^{R(G)}(\Z, R(H))$ coincide. Let us reduce this case to the proposition stated above.

Observe that the sequence $(\rho_1, \dots, \rho_l)$ is regular in $R(G)$ and that $\Z = R(G) / (\rho_1, \dots, \rho_l )$. Recall that for our varieties $R(H) = R(G) / I$ where $I = \Ker\,i^*$. Thus we get: 
$$\oplus_p E_{p, 0}^2 = \oplus_p \Tor_p^{R(G)}(R(G)/(\rho_1, \dots, \rho_l), R(G)/I).$$

A product on $\oplus_p \Lambda^p(\Z^m)$ admits a natural extension by applying the localization homomorphism $\oplus_p \Lambda^p(\Z^m) \hookrightarrow \oplus_p \Lambda^p(\Q^m)$. Passing to the localization allows to consider the graded ring $\oplus_p \Tor_p^R (R/\ag, R / J)$ in which the ideal $\ag$ is already maximal. By means of the identity
$\Tor_p^R (R/\ag, R / J) =  \linebreak \Tor_p^{R_{\ag}} (R_{\ag} / (\ag \cdot R_{\ag}), R_{\ag} / J_{\ag})$ the statement can be reduced to the case of a local ring $R$.
\end{proof}

Let us consider the edge homomorphisms $g_i \colon K_i(G/H) \to E_{i, 0}^2 = \Lambda^i(\Z^m)$. Since the differentials $d_{1, 0}^r$ are zero for every $r \geqslant 2$, we see that $E_{1, 0}^{\infty} = E_{1, 0}^2$ hence $g_1$ is surjective. The edge homomorphism is multiplicative with respect to the product $\smile_2$ \cite[Prop.~1.3]{lev}, i.e., $g_i(a) \smile_2 g_j(b) = g_{i+j}(a \cup b)$. It follows from Lemma \ref{multi} that the edge homomorphism is multiplicative with respect to the natural product on $\Lambda(\Z^m)$. The algebra $\Lambda(\Z^m)$ is generated by the component $\Lambda^1(\Z^m)$, thus surjectivity of $g_1$ implies that $g_i$ are surjective for every $i$. It follows from the surjectivity of the homomorphisms $g_i$ that $E_{i, 0}^2 = E_{i, 0}^{\infty}$. Therefore all the differentials $d_{p, 0}^r$ are zero for every $r \geqslant 2$. 

The Levine spectral sequence is a module over $K_*(F)$ \cite[Lemma~1.2]{lev}. Since 
$E_{p,q}^2 = E_{p,0}^2 \otimes K_q(F)$ and $d_{p, 0}^2 = 0$, all the differentials on the second page are zero. Using the facts that $d_{p, 0}^r = 0$ and that $E_{p,q}^r$ is a $K_*(F)$-module for every $r \geqslant 2$, we get that the differentials are zero on the higher pages also. As a result we see that the spectral sequence degenerates at the second page.

\begin{cl}
\label{filtration}
There is a filtration on $K_*(G/H)$ whose successive quotients are $K_*(F)$, $K_*(F)^m$, $\Lambda^2(K_*(F)^m), \dots, \Lambda^m(K_*(F)^m)$.
\end{cl}
\begin{proof}
Since $E_{p, q}^{\infty} = E_{p, q}^2 = \Lambda^p(\Z^m) \otimes K_q(F) \Longrightarrow K_{p+q}(G/H)$, there is a filtration on each $K_i(G/H)$ with the following successive quotients:
 $K_i(F), \linebreak K_{i-1}(F)^m, \, \Lambda^2 (K_{i-2}(F)^m), \, \dots, \, \Lambda^i(\Z^m).$
These filtrations give a general filtration on $K_*(G/H)$ with the desired successive quotients.
\end{proof}

\begin{cl}
\label{free}
$K_*(G/H)$ is a free $K_*(F)$-module of rank $2^m$.
\end{cl}
\begin{proof}
Let us consider the filtration on $K_*(G/H)$ defined in Corollary \ref{filtration}. All the successive quotients are free $K_*(F)$-modules of finite rank, therefore short exact sequences ending with those modules are split. It means that we have an isomorphism of $K_*(F)$-modules (which may not respect the graded structures):
$$
K_*(G/H) \simeq K_*(F) \oplus K_*(F)^m \oplus \Lambda^2(K_*(F)^m) \oplus \dots \oplus \Lambda^m(K_*(F)^m).
$$
\end{proof}

\subsection{Application of $E_{p,q}^*$}
Let us get some information about $K_*(G/H)$ using the considered spectral sequence.

\begin{lm}
\label{K1}
$K_1(G/H) \simeq K_1(F) \oplus \Z^m$. In particular, for reduced K-theory we have
 $ \widetilde K_1(G/H) \simeq \Z^m$.
\end{lm}
\begin{proof}
The filtration on $K_1(G/H)$ implies the existence of a short exact sequence:
$$
\xymatrix{
0 \ar[r] & K_1(F) \ar[r] & K_1(G/H) \ar[r] & \Z^m \ar[r]  & 0 \\
}
$$
It splits by means of a homomorphism $j^* \colon K_1(G/H) \to K_1(F)$ induced by an inclusion
$j \colon pt \hookrightarrow G/H$.
\end{proof}

Let us introduce some notation: $A$ is the graded ring $K_*(F)$; $A^+ = \bigoplus_{i > 0} A_i$ ($A / A^+ = \Z$); $B$ is the graded $A$-module $K_*(G/H)$. The quotient module $B / (A^+ \! \cdot \! B)$ has the structure of a $\Z$-module.

\begin{lm}
\label{reduction}
There is an isomorphism of abelian groups $ \,\, B / (A^+ \! \cdot \! B) \simeq \Lambda(\Z^m) $.
\end{lm}
\begin{proof}
For every $p>0$ there is a filtration on $K_p(G/H)$ of length $p+1$ such that $K_p(G/H)^{(p)} = K_p(F)$.  Taking the quotient by $K_p(G/H)^{(p)}$ we get a filtration of length $p$ on the quotient group, the first term of which is again zero in $ B / (A^+ \! \cdot \! B)$. Iterating this process we get that the homomorphism $B \to B / (A^+ \! \cdot \! B)$ sends the free summand $K_p(G/H)$ to $\Lambda^p(\Z^m) = E_{p, 0}^2$.
\end{proof}

\section{Computation of $K_*(G/H)$ as a graded $K_*(F)$-module}

Let us consider the exterior algebra $\Lambda(\Z^m)$ as an abelian group with the natural grading: $\Lambda(\Z^m)_i = \Lambda^i(\Z^m)$.

\begin{prop}
\label{main_th}
There is an isomorphism of graded $K_*(F)$-modules: 
$$K_*(G/H) \simeq K_*(F) \otimes_{\Z} \Lambda(\Z^m).$$
\end{prop}

\begin{proof}
Let $S$ be a graded ring, $S^+ = \bigoplus_{i > 0} S_i$. For a graded $S$-module $P$ we will denote the $S/S^+$-module $P / (S^+ \! \cdot \! P)$ as $\overline{P}$.

As earlier, we will write $A$ for $K_*(F)$. Let us introduce notation for graded $A$-modules: $B = K_*(G/H)$, $\,\,C = K_*(F) \otimes \Lambda(\Z^m)$, and let $j \colon A \hookrightarrow B$ be the canonical inclusion.

Consider the homomorphism of graded $A$-modules: $$\phi = j \otimes \Lambda(id) \colon \, C \to B,$$

where $\Lambda(id) = id \colon \Z^m \to \Z^m \subset B_1$ (see Lemma~\ref{K1}), and   $\Lambda(id)(e_{i_1} \wedge \dots \wedge e_{i_r}) = \Lambda(id)(e_{i_1}) \cup \dots \cup \Lambda(id)(e_{i_r}) \in B_r$. 

We will show that $\phi$ is an isomorphism. To do that we will use the graded version of the Nakayama lemma. 

\begin{lm*}[Graded Nakayama Lemma]
\label{nakayama}
\noindent
Let $R = \bigoplus_{i=0}^{\infty} R_i$ be a graded ring, $R^+ = \bigoplus_{i>0} R_i$. 
Let $M$ be a graded $R$-module such that $M_j = 0$ for $j << 0$. Then $R^+ \cdot M = M$ implies $M = 0$.
\end{lm*}

First we will check that $\phi$ is an epimorphism.

\begin{lm}
\label{sur}
The homomorphism $\phi$ is surjective.
\end{lm}
\begin{proof}
Observe that $\overline{C} = \Z \otimes \Lambda(\Z^m) \simeq \Lambda(\Z^m)$. It follows from Lemma~ \ref{reduction} that also $\overline{B} \simeq \Lambda(\Z^m)$. The induced homomorphism of $\Z$-modules $\overline{\phi} \colon \overline{C} \to \overline{B}$ maps $\Z^m$ to $\Z^m$ isomorphically. Thus $\overline{\phi}$ is an isomorphism. It implies that $\overline{\Coker\,\phi} = 0$. Then by the graded Nakayama lemma $\Coker\,\phi = 0$.
\end{proof}

It follows from Corollary \ref{free} and Lemma \ref{sur} that the homomorphism $\phi$ is a graded epimorphism of free finitely-generated $K_*(F)$-modules of the same rank. It implies that $\overline{\Ker\,\phi} = 0$, and so  $\Ker\,\phi = 0$ by the graded Nakayma lemma.  Thus $\phi$ is an isomorphism of graded $K_*(F)$-modules.

\end{proof}

\section{Computation of generators of $K_*(G/H)$ as $K_*(F)$-module}
\subsection{Computation of generators of $\widetilde K_1(G/H)$}
It follows from Lemma~\ref{K1} that for reduced K-theory we have $ \widetilde K_1(G/H) \simeq \Z^m$. To get the final answer we only need to find $m$ generating elements for $ \widetilde K_1(G/H)$. 
First let us consi\-der $ \widetilde K_1(G)$. It was proved by M.\,Levine \cite[Theorem~2.1 and Cor.~2.2]{lev} that for $G = \SL_{2n}$ and $G = \E_6$ there is an isomorphism $ \widetilde K_1(G) \simeq \Z^l$ where $l = rk(G)$; moreover, $ \widetilde K_1(G)$ is generated by the elements $[\rho_1], \dots, [\rho_l] \in K_1(G)$. 

Recall that for each pair of representations $(\rho_{i_1}, \rho_{i_2})$ we constructed an element ${\beta(\rho_{i_1}-\rho_{i_2}) \in K_1(G/H)}$ in~\ref{construction}.

\begin{prop}
\label{main_k1}
$\widetilde K_1(G/H)$ is generated by the elements $u_i = \beta(\rho_{i_1} - \rho_{i_2})$, ${1 \leqslant i \leqslant m}$.
\end{prop}
\begin{proof}
Let $\Z^m$ be generated by elements $e_1, \dots, e_m$ as a free abelian group. Consider the diagram of abelian groups and their homomorphisms:
$$
\xymatrix{
\Z^m \ar[r]^-{\psi}  & \widetilde K_1(G/H) \simeq \Z^m \ar[ld]^-{p^*} \\
\widetilde K_1(G) \simeq \Z^l  \ar[u]^{\chi}}
$$
The homomorphisms are defined the following way:

1) $\psi \colon e_i \mapsto u_i$,

2) $p^* \colon \widetilde K_1(G/H) \to \widetilde K_1(G)$ is induced by projection $p \colon G \to G/H$,

3) $\chi \colon \widetilde K_1(G) \to \Z^m$ is defined on generators: 
$
[\rho_k] \mapsto \begin{cases}
e_i, &\text{ if $k=i_1$,}\\
0,  &\text{else.}
\end{cases}
$

We will show that $\psi$ is an isomorphism. Observe that: 
$$p^*(u_i) = p^*([\rho_{i_1} \cdot \alpha_i^{-1} \cdot \rho_{i_2}^{-1} \cdot \alpha_i]) = 
[\rho_{i_1}] + [\alpha_i^{-1}] + [\rho_{i_2}^{-1}] + [\alpha_i] = [\rho_{i_1}] - [\rho_{i_2}].$$
Therefore, 
$$
(\chi \circ p^* \circ \psi) (e_i) =( \chi \circ p^* )(u_i) = \chi ([\rho_{i_1}] - [\rho_{i_2}]) = e_i.
$$
Hence $\chi \circ p^* \circ \psi = id$, i.e., $\psi$ is injective. 
Note that $\psi \colon \Z^m \to \Z^m$ has a left inverse $\chi \circ p^*$. It implies that $\Coker \, \psi = 0$ and so $\psi$ is surjective.

\end{proof}

\subsection{Final result}
To study later K-theory of twisted forms of varieties, we will formulate now the obtained result in functorial terms.

Let $X$ be a variety and let $\xi = (x_1, \dots, x_m)$ be a set of elements in $K_1(X)$. For every subset of indices $I = \{i_1 < \dots < i_q\} \subseteq \{1, \dots, m\}$ we denote $x_I = x_{i_1} \cup \dots \cup x_{i_q} \in
 K_{|I|}(X)$ where $|I|$ is cardinality of $I$. For $I = \varnothing$ define $x_{\varnothing} = 1 \in K_0(F)$.

Let us consider the homomorphisms of graded $K_*(F)$-modules:
$$
\Theta_{I, \xi} \colon K_{*-|I|}(F) \to K_*(X); \quad \alpha \mapsto x_I \cup \alpha.
$$
We define the homomorphism $\Theta_{\xi}$ the following way:
$$
\Theta_{\xi} = \sum_I \Theta_{I, \xi} \colon \bigoplus_I K_{*-|I|}(F) \to K_*(X),
$$
where $I$ runs through all the subsets of a set $\{1, \dots, m\}$.

$ $

The final result follows from Propositions \ref{main_th} and \ref{main_k1}.
\begin{tm}
\label{th_card}
Let $t = (t_1, \dots, t_{n-1})$ and $s = (s_1, s_2)$ be the sets of elements in $K_1(\SL_{2n}/\Sp_{2n})$ and $K_1(\E_6/\F_4)$ respectively, defined in~\ref{elements}. Then the homomorphisms of graded $K_*(F)$-modules $\Theta_t$ and $\Theta_s$ are isomorphisms:
$$
\Theta_t \colon \bigoplus_{I \subseteq \{1, \dots, n-1\}} \! \! \! \! \! K_{*-|I|}(F) \,\, \xrightarrow{\sim} \,\, K_*(\SL_{2n}/\Sp_{2n});
$$
$$
\Theta_s \colon \bigoplus_{I \subseteq \{1, 2\}} \! \! \! K_{*-|I|}(F) \,\, \xrightarrow{\sim} \,\, K_*(\E_6 / \F_4).
$$
\end{tm}

\section{Twisted forms and central simple algebras}

From now on we will assume that $\chr(F) \ne 2$. As earlier, we denote both varieties $\SL_{2n} / \Sp_{2n}$ and $\E_6 / \F_4$ as $G/H$ or $X$. We denote the center of an algebraic group $G$ as $Z(G)$.

Let us consider an action of the group $H$ on the variety $G/H$ by left translation. This action can be extended to $\overline{H} = H / Z(H)$ (in the first case $\overline{H} = \Sp_{2n} / \mu_2$, in the second case $\overline{H} = \F_4$). Let us fix a 1-cocycle $\gamma \colon \Gal(F^{sep} / F) \to \overline{H}(F^{sep})$. Since there is an action of $\overline H$ on $G/H$, we can consider a twisted form of the variety $X$ denoted $X_{\gamma}$. The rest of the paper consists of the computation of $K_*(X_{\gamma})$.

\subsection{Twisting of central simple algebras}
\label{twist_alg}
\begin{df}
For an algebraic group $G$ let us introduce a notation for the {\it group of characters of the center}: $\Ch(G) = \Hom(Z(G), \Gm)$.  
\end{df}
 
\begin{df}
A representation $\sigma \colon G \to \GL(V)$ of an algebraic group $G$ is called $\Ch$-\textit{homogeneous} if there is a character $\chi \in \Ch(G)$ such that $\sigma(z) v = \chi(z) \cdot v$ for every $z \in Z(G), \, v \in V$. In particular, irreducible representations are $\Ch$-homogeneous.
\end{df}

Let $\sigma \colon H \to \GL(V)$ be a $\Ch$-homogeneous representation of the group $H$ and $A = \End_{F}(V)$ be a central simple algebra (we will write $\End(V)$ for $\End_{F}(V)$).
Consider the action of $H$ on $A$ by conjugation: \linebreak $(h, \alpha) \mapsto \sigma(h) \alpha \sigma(h)^{-1}$. It induces an action of $\overline H$ on the algebra $A$. From the action of $\overline H$ on the algebra $A$ and a cocycle $\gamma \colon \Gal(F^{sep} / F) \to \overline{H}(F^{sep})$ the Tits construction allows to build a twisted algebra $A_{\gamma}^{\sigma}$ (see~\cite[\S3]{tits} or~\cite[8.6]{pan}).

\begin{rem}
Let $V$ be a $2n$-dimensional vector space over $F$, $A = \End(V)$, and let $\tau$ be an involution on $A$ corresponding to the standard antisymmetric form. Consider the action of $\overline{\Sp}_{2n}$ on $\SL_{2n}$ and $\Sp_{2n}$ by conjugation. Then for twisted forms of $\SL_{2n} / \Sp_{2n}$ we have: 
$$(\SL_{2n}/\Sp_{2n})_{\gamma} \simeq (\SL_{2n})_{\gamma}/(\Sp_{2n})_{\gamma} =(\SL_{1, A})_{\gamma} / \Sp(A, \tau)_{\gamma} = 
\SL_{1, A_{\gamma}} / \Sp(A_{\gamma}, \tau_{\gamma}).$$
\end{rem}

\subsection{Computation of central simple algebras}
\label{cpas}

Let $\sigma \colon H \to \GL(V)$ be a $\Ch$-homogeneous representation of a group $H$ and let $A = \End(V)$. The class of the algebra $A_{\gamma}^{\sigma}$ in the Brauer group $\Br(F)$ depends only on the character $\chi \in \Ch(H)$ representing the action of $Z(H)$ on $V$ under $\sigma$ \cite[8.7]{pan}. We will compute the equivalence classes of the algebras $A_{\gamma}^{\sigma}$ in the Brauer group for fundamental representations $\sigma$ of the groups $H = \Sp_{2n}$ and $H=\F_4$.

The center of the group $\Sp_{2n}$ is equal to $\mu_2$, so $\Ch(\Sp_{2n}) = \Z / 2\Z$. Under vector representation $V$ the center acts with the character $\overline{1} \in \Z / 2\Z$. Under representation $\Lambda^r V$ the center acts with the character $\overline{r} \in \Z / 2\Z$. Hence in $\Br(F)$ there are equivalences for $A_{i, \gamma} = \End(\Lambda^i V)_{\gamma}^{\Lambda^i V}$: 
$$A_{i, \gamma} \sim A_{\gamma} \text{ if $i$ is odd;  } \quad A_{i, \gamma} \sim F \text{ if $i$ is even,} $$
where $A = \End(V)$, $\, V$ is a $2n$-dimensional vector space, $i = 1 \dots n$.

The center of the group $F_4$ is trivial, so the group of characters is tri\-vial also. Therefore for all four algebras $A_{i, \gamma} = \End(V_i)_{\gamma}^{\sigma_i}$ corresponding to fundamental representations $\sigma_i$ of the group $F_4$ we have $A_{i, \gamma} \sim F$.

\section{Construction of certain elements in $K_1$}
\label{elem_twist}

\begin{df}
Let $B$ be a central simple $F$-algebra. For an affine variety $Y$ we put $B[Y] = B \otimes_{F} F[Y]$. Then $K_1$ with coefficients in $B$ is defined as follows:
$$
K_1(Y, B) = K_1(B[Y]) = \GL(B[Y]) / E(B[Y]).
$$
\end{df}

\subsubsection*{General construction}
Suppose there is a morphism $f \in \Mor_{F}(Y, \GL_{1, B})$. We identify $\Mor_{F}(Y, \GL_{1, B})$ with $\GL_1(B[Y])$ \cite[Section~9]{pan} and consider the composition:
$$
\xymatrix{
\GL_1(B[Y]) \ar@{^{(}->}[r] & \GL(B[Y]) \ar@{->>}[r] & K_1(B[Y]) \ar[r]^-{\sim} & K_1(B^{op}[Y]) \\ }
$$
This way we can assign an element $[f] \in K_1(Y, B^{op})$ to the morphism $f$.

\subsubsection*{Application}
Suppose that the representations $(\rho_{i_1}, \rho_{i_2})$ of the group $G$ (where notation is as in~\ref{defs}) act on a vector space $V_i$. Then each pair defines the map $\widetilde \rho_i \colon G/H \to \GL(V_i)$ described in~\ref{construction}:
$$\widetilde \rho_i \colon gH \mapsto \rho_{i_1}(g) \alpha_{i}^{-1} \rho_{i_2}(g)^{-1} \alpha_{i}, $$ 
where $\alpha_i$ satisfy $\rho_{i_1}(h) = \alpha_{i}^{-1} \rho_{i_2}(h) \alpha_{i}$ for every $h \in H$.

Consider the action of $\overline H$ on $\GL(V_i)$: $(h, \chi) \mapsto \rho_{i_1}(h) \chi  \rho_{i_1}(h)^{-1}$. Then $\widetilde \rho_i$ are $\overline H$-equivariant morphisms (with respect to the action by left translation of $\overline H$ on $G/H$). 

Denote $A_i = \End(V_i)$. Observe that representations $\rho_{i_1} \colon H \to \GL(V_i)$ are $\Ch$-homogeneous. Therefore we can twist $\GL(V_i) = \GL_{1, A_i}$ with a \linebreak 1-cocycle $\gamma \colon \Gal(F^{sep} / F) \to \overline{H}(F^{sep})$ (see~\ref{twist_alg}). We will write $A_{i, \gamma}$ for $A_{i, \gamma}^{\rho_{i_1}}$. Furthermore, we can twist with this cocycle $X = G/H$ and $\widetilde \rho_i$ (because of $\overline H$-equivariance of morphisms $\widetilde \rho_i$). We get the following objects:
$$
(G/H)_{\gamma};\, \, \GL(V_i)_{\gamma} = GL_{1, A_{i, \gamma}}; \, \, 
\widetilde \rho_{i, \gamma} \colon (G/H)_{\gamma} \to \GL_{1, A_{i, \gamma}},
$$
where $\widetilde \rho_{i, \gamma} \in \Mor_F(X_{\gamma}, \GL_{1, A_{i, \gamma}})$. 
To the morphism $\widetilde  \rho_{i, \gamma}$ we assign the element $[\widetilde \rho_{i, \gamma}] \in K_1(X_{\gamma},  A_{i, \gamma}^{op})$ the way described in the general construction. After fixing a cocycle $\gamma$ the corresponding elements of the K-theory will be denoted $[\widetilde t_1], \dots, [\widetilde t_{n-1}]$ in the case of the variety $(\SL_{2n}/\Sp_{2n})_{\gamma}$ and $[\widetilde s_1], [\widetilde s_2]$ in the case of the variety $(\E_6 / \F_4)_{\gamma}$.

Recall that we know the equivalence classes of the algebras $A_{i, \gamma}$ in the Brauer group from~\ref{cpas}. Since $K_1(Y, F^{op}) = K_1(Y)$ for every variety $Y$, we have:
\begin{gather*}
[\widetilde t_i] \in K_1((\SL_{2n}/\Sp_{2n})_{\gamma}, A_{\gamma}^{op}) \text{ if $i$ is odd,} \\
[\widetilde t_i] \in K_1((\SL_{2n}/\Sp_{2n})_{\gamma}) \text{ if $i$ even;}  \\
[\widetilde s_1], [\widetilde s_2] \in K_1((\E_6 / \F_4)_{\gamma}),
\end{gather*}
where $0 \leqslant i \leqslant n-1$, $V$ is a $2n$-dimensional $F$-vector space, $A = \End(V)$.

\section{Computation of K-theory of twisted forms}

Let $Y$ be an affine variety, $B_1, \dots, B_m$ central simple $F$-algebras and $\xi = (x_1, \dots, x_m)$ a set of elements such that ${x_i \in K_1(Y, B_i^{op})}$. For every subset $I = \{i_1 < \dots < i_q\} \subseteq \{1, \dots, m\}$ we denote $x_I = x_{i_1} \cup \dots \cup x_{i_q} \in K_{|I|}(Y, B_{i_1}^{op} \otimes \dots \otimes B_{i_q}^{op})$.

Define $B_I = B_{i_1} \otimes \dots \otimes B_{i_q}$ and consider the homomorphism of graded $K_*(F)$-modules:
$$
\Theta_{I, \xi} \colon K_{*-|I|}(B_I) \to K_*(Y); \quad \alpha \mapsto x_I \cup_{B_I} \alpha.
$$

We define the homomorphism $\Theta_{\xi}$ the following way:
$$
\Theta_{\xi} = \sum_I \Theta_{I, \xi} \colon \bigoplus_I K_{*-|I|}(B_I) \to K_*(Y),
$$
where $I$ runs through all subsets of the set $\{1, \dots, m\}$.

For the variety $X_{\gamma} = (G/H)_{\gamma}$ we take central simple algebras $B_i$ equal to $A_{i, \gamma} = \End(V_i)_{\gamma}^{\rho_{i_1}}, \, i=1, \dots, m$, where $V_i$ is the vector space which the representations $\rho_{i_1}$ and $\rho_{i_2}$ of the group $G$ act on. We consider the set of elements $\widetilde \rho = ([\widetilde \rho_1], \dots, [\widetilde \rho_m])$ where $[\widetilde \rho_i] \in K_1(X_{\gamma},  A_{i, \gamma}^{op})$ (see~\ref{elem_twist}). This way, we can define the homomorphism:
$$\Theta_{\widetilde \rho} \colon \bigoplus \limits_{I \subseteq \{1, \dots, m\}} \! \! \! \! \! K_{*-|I|}(B_I) \,\, \to \,\, K_*((G/H)_{\gamma}).
$$
Panin's splitting principle tells us~\cite[Theorem~1.1]{pan} that in order to prove that the homomorphism $\Theta_{\widetilde \rho}$ is an isomorphism, it is enough to check the following property.

\begin{prop}
\label{vazhnoe}
Let $F \subset E$ be any field extension such that cocycle $\gamma_{E}$ is a coboundary. Then the homomorphism $\Theta_{\widetilde \rho}$ after scalar extension up to the field $E$ becomes an isomorphism:
$$ \Theta_{\widetilde \rho}(E) \colon \bigoplus \limits_{I \subseteq \{1, \dots, m\}} \! \! \! \! \! 
K_{*-|I|}(B_I \otimes_{F} E) \,\,  \xrightarrow{\sim} \,\, K_*((G/H)_{\gamma} \times Spec \, E).
$$
\end{prop}

\begin{proof}
Since $\gamma_{E}$ is trivial, all the twistings trivialize over the field $E$:
\begin{gather*} 
(G/H)_{\gamma} \times Spec \, E \simeq (G/H)_{E}, \\
A_{i, \gamma} \otimes E \simeq A_i \otimes E \sim E \text{ (equivalence in $\Br(E)$)}, \\
[\widetilde \rho_i] \otimes E = t_{i, E} \text{ if $G/H = \SL_{2n}/\Sp_{2n}$}, \\
[\widetilde \rho_i] \otimes E = s_{i, E} \text{ if $G/H = \E_6 / \F_4$},
\end{gather*}
where $t_i$ and $s_i$ are defined the same way as in Theorem~\ref{th_card}.
We see that the homomorphism $\Theta_{\widetilde \rho}(E) $ in case of every considered variety $G/H$ coincides with the corresponding isomorphism from Theorem~\ref{th_card} after scalar extension up to the field $E$.
\end{proof}

Therefore for varieties $\SL_{2n, \gamma} / \Sp_{2n, \gamma}$ as well as for varieties $(\E_6 / \F_4)_{\gamma}$ the homomorphism $\Theta_{\widetilde \rho}$ is an isomorphism. It implies the final result.
\begin{tm}
Assume $\chr(F) \ne 2$. Let $\gamma \colon \Gal(F^{sep} / F) \to (\Sp_{2n} / \mu_2)(F^{sep})$ be a 1-cocylce. Let $\widetilde t = ([\widetilde t_1], \dots, [\widetilde t_{n-1}])$ be the set of elements defined in~\ref{elem_twist}, $A = \End(V)$ where $V$ is a $2n$-dimensional vector space over $F$, and let $\tau$ be the standard symplectic involution on $A$. Denote $B_i$ the central simple algebra $A_{\gamma}$ for $i$ odd, and $F$ for $i$ even $(1 \leqslant i \leqslant n-1)$.
 Then the homomor\-phism $\Theta_{\widetilde t}$ is an isomorphism of graded $K_*(F)$-modules:
$$
\Theta_{\widetilde t} \colon \bigoplus_{I \subseteq \{1, \dots, n-1\}} \! \! \! \! \! K_{*-|I|}(B_I) \,\, \xrightarrow{\sim} \,\, K_*((\SL_{2n}/\Sp_{2n})_{\gamma}) = K_*(\SL_{1, A_{\gamma}} / \Sp(A_{\gamma}, \tau_{\gamma}) ).
$$
Let $\delta \colon \Gal(F^{sep} / F) \to \F_4(F^{sep})$ be a 1-cocycle. Let $\widetilde s = ([\widetilde s_1], [\widetilde s_2])$ be the set of elements defined in~\ref{elem_twist}. Then the homomorphism $\Theta_{\widetilde s}$ is an isomorphism of graded $K_*(F)$-modules:
$$
\Theta_{\widetilde s} \colon \bigoplus_{I \subseteq \{1, 2\}} \! \! \! K_{*-|I|}(F) \,\, \xrightarrow{\sim} \,\, K_*((\E_6 / \F_4)_{\delta}).
$$
\end{tm}

\end{document}